\documentclass[a4paper, pdftex]{amsart}

\usepackage[utf8]{inputenc}
\usepackage[T1]{fontenc}
\usepackage[pdftex]{graphicx}
\usepackage{amsmath}
\usepackage{amsfonts}
\usepackage{amsthm}
\usepackage{amssymb}
\usepackage{setspace}
\usepackage{caption}
\usepackage{thmtools}
\usepackage{thm-restate}
\usepackage{mathtools}
\usepackage{hyperref}
\usepackage[capitalize]{cleveref}
\usepackage{verbatim}
\usepackage[all, cmtip]{xy}
\usepackage{multirow}
\usepackage{bm}
\usepackage{colonequals}
\usepackage{enumitem}
\usepackage{permute}
\usepackage[textsize = tiny]{todonotes}

\begin{document}
\newtheorem*{thm*}{Theorem}
\newtheorem{theorem}{Theorem}[section]
\newtheorem{corollary}[theorem]{Corollary}
\newtheorem{lemma}[theorem]{Lemma}
\newtheorem{fact}[theorem]{Fact}
\newtheorem*{fact*}{Fact}
\newtheorem{proposition}[theorem]{Proposition}
\newtheorem{claim}[theorem]{Claim}

\newcounter{theoremalph}
\renewcommand{\thetheoremalph}{\Alph{theoremalph}}
\newtheorem{thmAlph}[theoremalph]{Theorem}
\theoremstyle{definition}
\newtheorem{question}[theorem]{Question}
\newtheorem{definition}[theorem]{Definition}

\theoremstyle{remark}
\newtheorem{remark}[theorem]{Remark}
\newtheorem{observation}[theorem]{Observation}
\newtheorem{example}[theorem]{Example}


	\renewcommand{\ll}{\left\langle}
	\newcommand{\rr}{\right\rangle}
	\newcommand{\ls}{\left\{}
	\newcommand{\rs}{\right\}}
	\newcommand{\sm}{\setminus}
	
	\newcommand{\mcA}{\ensuremath{\mathcal{A}}}
	\newcommand{\mcB}{\ensuremath{\mathcal{B}}}
	\newcommand{\mcC}{\ensuremath{\mathcal{C}}}
	\newcommand{\mcE}{\ensuremath{\mathcal{E}}}
	\newcommand{\mcF}{\ensuremath{\mathcal{F}}}
	\newcommand{\mcG}{\ensuremath{\mathcal{G}}}
	\newcommand{\mcH}{\ensuremath{\mathcal{H}}}
	\newcommand{\mcI}{\ensuremath{\mathcal{I}}}
	\newcommand{\mcJ}{\ensuremath{\mathcal{J}}}
	\newcommand{\mcN}{\ensuremath{\mathcal{N}}}
	\newcommand{\mcO}{\ensuremath{\mathcal{O}}}
	\newcommand{\mcP}{\ensuremath{\mathcal{P}}}
	\newcommand{\mcQ}{\ensuremath{\mathcal{Q}}}
 	\newcommand{\mcR}{\ensuremath{\mathcal{R}}}
	\newcommand{\mcS}{\ensuremath{\mathcal{S}}}
	\newcommand{\mcU}{\ensuremath{\mathcal{U}}}
	\newcommand{\mcV}{\ensuremath{\mathcal{V}}}
	\newcommand{\mcW}{\ensuremath{\mathcal{W}}}
	\newcommand{\mcZ}{\ensuremath{\mathcal{Z}}}
	\newcommand{\mbQ}{\ensuremath{\mathbb{Q}}}
	\newcommand{\mbZ}{\ensuremath{\mathbb{Z}}}
    \newcommand{\mbN}{\ensuremath{\mathbb{N}}}

	\newcommand{\on}[1]{\operatorname{#1}}
	\newcommand{\set}[1]{\ensuremath{ \left\lbrace #1 \right\rbrace}}	
	\newcommand{\grep}[2]{\ensuremath{\left\langle #1 \, \middle| \, #2\right\rangle}}
	\newcommand{\real}[1]{\ensuremath{\left\lVert #1\right\rVert}}
	\newcommand{\overbar}[1]{\mkern 2mu\overline{\mkern-2mu#1\mkern-2mu}\mkern 2mu}
	
	\newcommand{\Aut}{\ensuremath{\operatorname{Aut}}}
	\newcommand{\AutO}{\ensuremath{\operatorname{Aut}^0}}
	\newcommand{\Out}{\ensuremath{\operatorname{Out}}}
	\newcommand{\Outo}[1][A_G]{\ensuremath{\operatorname{Out}^0(#1)}}
	\newcommand{\Stab}{\operatorname{Stab}}
	\newcommand{\Sym}{\operatorname{Sym}}
	\newcommand{\St}{\operatorname{St}}
	\newcommand{\Stsymp}{\operatorname{St}^\omega}
	
	\newcommand{\op}{\operatorname{op}}
	\newcommand{\st}{\operatorname{star}}
	\newcommand{\lk}{\operatorname{lk}}
	\newcommand{\im}{\operatorname{im}}
	\newcommand{\rk}{\operatorname{rk}}
	\newcommand{\corank}{\operatorname{crk}}
	\newcommand{\spn}{\operatorname{span}}
	\newcommand{\Ind}{\operatorname{Ind}}

	\newcommand{\GL}[2]{\ensuremath{\operatorname{GL}_{#1}(#2)}}
	\newcommand{\SL}[2]{\ensuremath{\operatorname{SL}_{#1}(#2)}}
	\newcommand{\Sp}[2]{\ensuremath{\operatorname{Sp}_{#1}(#2)}}
	
	\newcommand{\class}{\operatorname{cl}}
	\newcommand{\Hom}{\operatorname{Hom}}
	\newcommand{\buildingssymp}{\operatorname{T}^{\omega}}
    \newcommand{\buildingssln}{\operatorname{T}}
    \newcommand{\building}{\Delta}
	\newcommand{\Rn}{R^{2n}}
    \newcommand{\Kn}{K^{2n}}

    \newcommand{\hcl}{\operatorname{cl}(R)}
    \newcommand{\dual}[1]{#1^{'}}
    \newcommand{\signp}[1][n]{S^B_{#1}} 
    \newcommand{\sat}[1]{\operatorname{Sat}(#1)}
    \newcommand{\len}{\operatorname{len}}
    \newcommand{\classconst}{\kappa}
    \newcommand{\apartment}{\Sigma}

    \newcommand{\symf}[2]{\ensuremath{\operatorname{\omega}(#1,#2)}}
    
    \newcommand{\chevalley}{\ensuremath{\mathcal{G}}}
    \newcommand{\productSL}{\ensuremath{\mathcal{S}}}

\author{Benjamin Br\"uck}
\address{Department for Mathematical Logic and Foundational Research, M\"unster University, Germany}
\email{benjamin.brueck@uni-muenster.de}

\author{Zachary Himes}
\address{Department of Mathematics, University of Michigan, 530 Church St, Ann Arbor, MI, 48109, USA}
\email{himesz@umich.edu}


\title[Top-degree cohomology in the symplectic group]{Top-degree rational cohomology in the symplectic group of a number ring}

\begin{abstract}
Let $K$ be a number field with ring of integers $R = \mathcal{O}_K$.
We show that if $R$ is not a principal ideal domain, then the symplectic group $\Sp{2n}{R}$ has non-trivial rational cohomology in its virtual cohomological dimension.
This demonstrates a sharp contrast to the situation where $R$ is Euclidean.
To prove our result, we study the symplectic Steinberg module, i.e.~the top-dimensional homology group of the spherical building associated to $\Sp{2n}{K}$. We show that this module is not generated by integral apartment classes.
Both of these results following from a vanishing theorem for homology with Steinberg coefficients.
\end{abstract}

\maketitle

\section{Introduction}
Let  $R$ be a Dedekind domain with field of fractions $K$. A \emph{symplectic form} on $R^{2n}$ is a non-degenerate bilinear form $\omega\colon R^{2n}\times R^{2n}\to R$ such that $\symf{x}{x}=0$ for all $x\in R^{2n}$.
The \emph{symplectic group} $\Sp{2n}{R}$ is the subgroup of all elements of $\GL{2n}{R}$ that preserve $\omega$.
The form on $R^{2n}$ naturally induces a symplectic form on $K^{2n}$, which we also denote by $\omega$. This gives an inclusion $\Sp{2n}{R}\subseteq \Sp{2n}{K}$.

Suppose now that $R = \mcO_K$ is the ring of integers in a number field $K$. Then $\Sp{2n}{R}$ is an arithmetic group and the rational cohomology of the symplectic group $H^{*}(\Sp{2n}{R}; \mathbb{Q})$ is finitely generated in all degrees \cite[Corollary 3]{MR0230332}.
Borel--Serre \cite{BS:Cornersarithmeticgroups} proved that there is
a non-negative integer $\nu_{n}\in \mathbb{Z}$, the \emph{virtual cohomological dimension} of $\Sp{2n}{R}$, such that for every $\Sp{2n}{R}$-module $M$ and for all $i> \nu_{n}$, we have 
$H^{i}(\Sp{2n}{R}; M)=0$.\footnote{A formula for $\nu_{n}$ in terms of $n$ and $R$ can be deduced from \cite[Corollary 11.4.3]{BS:Cornersarithmeticgroups}. For example, when $R=\mbZ$, we have $\nu_{n}=n^{2}$.}
In this paper, we study $H^{\nu_{n}}(\Sp{2n}{R}; \mathbb{Q})$.

If $R$ is Euclidean, it is known that these groups vanish:
Br\"uck--Patzt--Sroka showed that $H^{\nu_{n}}(\Sp{2n}{\mathbb{Z}}; \mathbb{Q})=0$ for all $n\geq 1$ \cite[Theorem 44]{Sroka2021}. This extended low-degree calculations by Igusa \cite{Igusa1962}, Lee-Weintraub \cite{Lee1985}, Hain \cite{Hain2002} and Hulek--Tommasi \cite{Hulek2012}. 
Br\"uck--Santos Rego--Sroka generalized this by showing that for all Euclidean number rings $R$,  $H^{\nu_{n}}(\Sp{2n}{R}; \mathbb{Q})=0$ for all $n\geq 1$ \cite[Theorem 1.1]{Brueck2022c}.

The following theorem shows that the situation is quite different outside the Euclidean setting. Let $\hcl$ denote the ideal class group of $R$. Recall that $R$ is a principal ideal domain (PID) if and only if $|\hcl|=1$ and that every Euclidean domain is a PID.

\begin{theorem}
\label{thm_non_trivial_cohomology}
Let $K$ be a number field and $R= \mcO_K$ its ring of integers. Then for all $n\geq 1$,
\begin{equation*}
\dim (H^{\nu_{n}}(\Sp{2n}{R}; \mbQ)) \geq (|\class(R)|-1)^n.
\end{equation*}
\end{theorem}

This result is motivated by a related lower bound for the rational cohomology of $\SL{n}{R}$ in its virtual cohomological dimension that was obtained by Church--Farb--Putman. In \cite[Theorem D]{Church2019}, they proved that when $R$ is a number ring, the dimension of $H^{\on{vcd}(\SL{n}{R})}(\SL{n}{R}; \mathbb{Q})$ is at least $(\vert\hcl\vert -1)^{n-1}$.
The general strategy we use to prove \cref{thm_non_trivial_cohomology} is closely related to theirs. In what follows, we give an outline of it in our setting.

\subsection{Proof sketch for \cref{thm_non_trivial_cohomology}}
A duality result by Borel--Serre (see \cref{thm_Borel_Serre}) implies that there is an isomorphism
\begin{equation}
\label{eq_BS_special_case}
	H^{\nu_{n}}( \Sp{2n}{R};\mbQ) \cong H_0(\Sp{2n}{R}; \Stsymp \otimes_{\mbZ} \mbQ),
\end{equation}
where $\Stsymp = \Stsymp_n(K)$ is the $(n-1)$-st reduced homology of $\buildingssymp(K^{2n})$, the poset of non-trivial isotropic subspaces\footnote{That is, subspaces $V\subseteq K^{2n}$ on which $\omega$ is trivial: $\omega(x,y)=0$ for all $x,y\in V$.} of $K^{2n}$.
The order complex of $\buildingssymp(K^{2n})$ is the \emph{spherical building of type $\mathtt{C}_n $ over $K$}. As elements of $\Sp{2n}{R}$ send isotropic subspaces of $K^{2n}$ to isotropic subspaces, the building $\buildingssymp(K^{2n})$ and hence its homology $\Stsymp$ come with an action of $\Sp{2n}{R}$. We call $\Stsymp$ the \emph{symplectic Steinberg module}.
Using \cref{eq_BS_special_case}, \cref{thm_non_trivial_cohomology} quickly follows from the following, which is the main result of the present article.

\begin{theorem}
\label{thm_steinberg_coinvariants}
Let $R$ be a Dedekind domain. Then for all $n\geq 1$,
\begin{equation*}
	\rk_{\mbZ} H_0(\Sp{2n}{R}; \Stsymp) \geq (|\class(R)|-1)^n.
\end{equation*}
\end{theorem}

$H_0(\Sp{2n}{R}; \Stsymp)$ is isomorphic to the $\Sp{2n}{R}$-coinvariants of $\Stsymp$.
We prove \cref{thm_steinberg_coinvariants} by obtaining a lower bound on the rank of these coinvariants.
In order to do so, we use a finite simplicial complex $X_{n}(\hcl)$
whose geometric realization is homotopy equivalent to a wedge of $(\vert\hcl\vert-1)^{n}$ many $(n-1)$-spheres (see \cref{sec_defn_X_n}). There is an $\Sp{2n}{R}$-invariant map
\begin{equation*}
	\psi\colon \buildingssymp(\Kn) \to X_{n}(\class(R))
\end{equation*}
that induces a map
\begin{equation*}
	\psi_{*}\colon \Stsymp_{\Sp{2n}{R}} \to \tilde{H}_{n-1}(X_{n}(\hcl))\cong \mbZ^{(|\class(R)|-1)^n}.
\end{equation*}
Proving that this induced map is onto gives the desired lower bound for the rank of $H_0(\Sp{2n}{R}; \Stsymp)$ and is the main difficulty in the proof.

For showing the surjectivity of $\psi_{*}$, we use explicit generating sets for $\Stsymp$ and $\tilde{H}_{n-1}(X_{n}(\hcl))$: 
We call a collection $\mathbf{L} = \ls L_1, \ldots, L_n, L_{-n}, \ldots, L_{-1}\rs$ of dimension-1 subspaces a \emph{symplectic frame of $K^{2n}$} if $L_i + L_j$ is isotropic if and only if $j\not=-i$. 
To each symplectic frame $\mathbf{L}$, one can associate a subcomplex of $\buildingssymp(K^{2n})$ called the \emph{apartment} $\apartment_{\mathbf{L}}$ corresponding to $\mathbf{L}$ (see \cref{sec_buildings}). This apartment is homeomorphic to an $(n-1)$-sphere. Its fundamental class determines an \emph{apartment class}
\begin{equation*}
	[\apartment_{\mathbf{L}}]\in \tilde{H}_{n-1}(\buildingssymp(K^{2n})) = \Stsymp.
\end{equation*}
The Solomon--Tits Theorem \cite{Sol:Steinbergcharacterfinite} states that the set of apartment classes generates $\Stsymp$.
Similarly, $\tilde{H}_{n-1}(X_{n}(\hcl))$ is generated by certain explicitly described classes $[\apartment_\mathbf{S}]$.
We show that $\psi_{*}$ is onto by constructing for each such generator $[\apartment_\mathbf{S}]$ an apartment class $[\apartment_\mathbf{L}]$ mapping to $[\apartment_\mathbf{S}]$ (for the idea behind this, see \cref{sec_psi_induces_surjectivity}).

\subsection{Non-integrality}
We call a symplectic frame $\mathbf{L} = \ls L_1, \ldots, L_n, L_{-n}, \ldots, L_{-1}\rs$ of $K^{2n}$ \emph{integral} if there is a symplectic basis $e_1, \ldots, e_n, e_{-n}, \ldots, e_{-1}$ of $R^{2n}$ such that for each $i$, the line $L_i$ is spanned by $e_i$.
We then call $\apartment_{\mathbf{L}}$ an integral apartment. 
These play a distinguished role because while there are in general several $\Sp{2n}{R}$-orbits of apartments, all \emph{integral} apartments lie in the same $\Sp{2n}{R}$-orbit.
In particular, if $\Stsymp$ is generated by integral apartment classes, then it is cyclic as an $\Sp{2n}{R}$-module. 
The question of whether this property holds was studied in several related cases, see e.g.~\cite{AR:modularsymbolcontinued, Brueck2023, Brueck2022c, Church2019, Gun:Symplecticmodularsymbols, MPWY:Nonintegralitysome,  Toth2005}.
In particular, T\'oth \cite{Toth2005} proved that it holds for $\Sp{2n}{R}$ if $R$ is Euclidean.
As a second consequence of \cref{thm_steinberg_coinvariants}, we show that it does not hold if $R$ is not a PID:

\begin{theorem}
\label{thm_non_integrality}
 Let $R$ be a Dedekind domain with fraction field $K$. If $2\leq \vert\hcl\vert<\infty$, then $\Stsymp= \Stsymp_n(K)$ is not generated by integral apartment classes.
\end{theorem}
A similar result was obtained by Church--Farb--Putman in the setting of $\SL{n}{R}$ in \cite[Theorem B']{Church2019}. Our proof however is different from theirs. Instead of using explicit combinatorics for symplectic groups, we obtain \cref{thm_non_integrality} directly from  \cref{thm_steinberg_coinvariants} by invoking a result of Br\"uck--Santos Rego--Sroka \cite{Brueck2022c}.

\subsection{Relation to $\on{SL}_n$ and other types of groups}
While the general strategy of proof for \cref{thm_non_trivial_cohomology} and \cref{thm_non_integrality} carries over from the $\on{SL}_n$ setting considered in \cite{Church2019} to the setting of $\on{Sp}_{2n}$ considered here, the different combinatorics of the associated spherical buildings leads to some complications. In the $\on{SL}_n$ setting, apartments are barycentric subdivisions of the boundary of an $(n-1)$-simplex, whereas in the $\on{Sp}_{2n}$ setting, they are barycentric subdivisions of the boundary of an $n$-dimensional cross polytope.
In particular, the proof that $\psi_{*}$ is surjective is slightly more intricate here.
The reason is that for constructing an apartment in the symplectic building $\buildingssymp$ (or equivalently, a symplectic frame $\mathbf{L}$), one needs to verify an isotropy condition on $L_{i}+ L_{j}$ among all pairs of elements $L_{i}, L_{j}$ in $\mathbf{L}$, see also \cref{sec_motivation}.
The different combinatorics is also shown by the fact that we need to replace the symmetric group, which describes the ``type $\mathtt{A}$'' combinatorics of apartments in the setting of $\on{SL}_n$, by the group of signed permutations, which describes the ``type $\mathtt{C}$'' combinatorics of apartments in the symplectic setting.

For comments on how this strategy could be extended to further types of groups, see \cref{sec_open_questions}.

\subsection{Outline}
\cref{sec_basics} contains background material. \cref{sec_surjectivity_symplectic} is the technical core of this article. In it, we prove that the map $\psi_*\colon \Stsymp_{\Sp{2n}{R}} \to \mbZ^{(|\class(R)|-1)^n}$ mentioned above is surjective (\cref{prop_homology_surjectivity}).
From this, we deduce our main results (\cref{thm_non_trivial_cohomology}, \cref{thm_steinberg_coinvariants} and \cref{thm_non_integrality}) in \cref{sec_proofs_main_results}.
We conclude with open questions and future directions in \cref{sec_open_questions}.

\subsection{Acknowledgments}
We would like to thank Jeremy Miller for many helpful conversations, for comments on a draft of this paper, and for suggesting a collaboration on this project in the first place.
We are grateful to Robin J.~Sroka for exchanges and his contributions in early stages of this project.
We thank Oscar Randal-Williams for many helpful conversations, for supporting a visit of the first named author in Cambridge, and for the crucial observation that the proof of \cref{thm_non_trivial_cohomology} does not need a stronger version of \cref{chain of istropics}. 
We thank Marc Burger, Matthew Cordes and Emmanuel Kowalski for interesting conversations about the topics of this article.
We thank the anonymous referees for helpful feedback that led to several improvements in the paper.

BB was partially funded by the Deutsche Forschungsgemeinschaft (DFG, German Research Foundation) under Germany's Excellence Strategy EXC 2044 – 390685587, Mathematics Münster: Dynamics–Geometry–Structure.
ZH was supported by the ERC under the European Union’s
Horizon 2020 research and innovation programme (grant agreement No. 756444).

\section{Basics}
\label{sec_basics}
Throughout this section, $R$ is a Dedekind domain and $K = \on{frac}(R)$ is its field of fractions.

\subsection{Rings, modules, class group}
\label{sec_basics_rings_modules}
We recall some basic facts about projective modules over Dedekind domains, see also \cite[in particular Section 2.2.1]{Church2019} and \cite{Mil:IntroductionalgebraicK}.

\subsubsection{Projective modules and direct summands} 
Let $M$ be a finitely generated projective $R$-module.
The \emph{rank} of $M$ is defined to be 
\begin{equation*}
	\rk(M)\colonequals\text{dim}_{K}(M\otimes_{R} K).
\end{equation*}
If $U\subseteq M$ is a submodule, it is projective again. We say that $U$ is a \emph{(direct) summand} if there exists a submodule $U'\subseteq M$ such that $M= U\oplus U'$. 

The intersection of direct summands is again a summand.
Furthermore, if $U, U'\subseteq M$ are summands and $U\subseteq U'$, then $U$ is a summand of $U'$.

\begin{lemma}
\label{iso btw summands and subspaces} 
The assignment $U\mapsto U\otimes_{R}K$ defines an isomorphism of posets
\begin{equation*}
	\{\text{summands of }M,\, \subseteq \}\longrightarrow \{\text{subspaces of } M\otimes_{R}K,\, \subseteq\},
\end{equation*}
which for each $k$ restricts to a bijection between the set of rank-$k$ summands and dimension-$k$ subspaces. The inverse is given by $V\mapsto V\cap M$ (here, $M$ is included into $M\otimes_R K$ via the identification $M\cong (M\otimes_R R)\subset M\otimes_R K$).
\end{lemma}
\begin{proof}
That this map defines a bijection is a standard fact, stated in \cite[Lemma 2.4(c)]{Church2019}. That it preserves the containment relations on summands and subspaces follows immediately from the definitions. 
\end{proof}

\subsubsection{The class group} The \emph{class group} $\hcl$ of $R$ is the set of isomorphism classes of rank-$1$ projective $R$-modules. The tensor product gives $\hcl$ the structure of an abelian group, which we write additively. For every rank-$n$  non-zero projective $R$-module $M$, there is a unique rank-$1$ projective $R$-module $I$ such that $M\cong R^{n-1}\oplus I$;
we write $[M]\coloneqq I\in\hcl$.
For two finite rank non-zero projective $R$-modules $M$ and $M'$, we have $[M\oplus M']=[M]+[M']$. Two such modules $M$ and $M'$ are isomorphic if and only if $\rk(M)=\rk(M')$ and $[M]=[M']$.

Given a non-zero subspace $V\subseteq \Kn$, we write $[V]$ to denote $ [V\cap \Rn]\in\hcl$. Given two disjoint non-zero subspaces $V, W\subseteq \Kn$, we do \emph{not} necessarily have that $[V\oplus W]=[V]+[W]$ since, in general, $(V\oplus W)\cap \Rn\neq (V\cap \Rn)\oplus (W\cap \Rn)$.
\subsection{Symplectic linear algebra}
In this article, we study modules that are equipped with a symplectic form.
\subsubsection{Basic definitions}
Given an $R$-module $M$, a bilinear form $\omega\colon M\times M\to R$  is \emph{non-degenerate} if the maps
    $$x\mapsto \symf{x}{-},\quad\quad y\mapsto \symf{-}{y} $$
from $M$ to the dual module $M^{*}=\text{Hom}_{R}(M, R)$ are bijections.   
A \emph{symplectic form} on $M$ is a non-degenerate bilinear form $\omega$ such that $\symf{x}{x}=0$ for all $x\in M$.
Such a form is skew-symmetric because $\symf{x}{y}+\symf{y}{x}=\symf{x+y}{x+y}-\symf{x}{x}-\symf{y}{y}=0$. A \emph{symplectic module (over $R$)} is a pair $(M, \omega)$, where $M$ is an $R$-module and $\omega$ is a symplectic form on $M$.
By abuse of notation, we often write $M$ for a symplectic module $(M, \omega)$. 

An example of a symplectic form $\omega_{0}$ on $R^{2n}$ comes from taking a basis $\allowbreak e_{1},\ldots, e_{n}$, $e_{-n},\ldots, e_{-1}$ of $R^{2n}$, setting $\omega_{0}(e_{i}, e_{j})=\omega_{0}(e_{j}, e_{i})=0$ for all $i$ and $j$ with $j\neq -i$ and $\omega_{0}(e_{i}, e_{-i})=-\omega_{0}(e_{-i}, e_{i})=1$ for all $i=1,\ldots, n$, and extending $\omega_{0}$ linearly.
We then call $e_{1},\ldots, e_{n}$, $e_{-n},\ldots, e_{-1}$ a \emph{symplectic basis} of $R^{2n}$ (with respect to $\omega_0$).
The following lemma allows us to assume that every symplectic module $M$ over $R$ is of the form $(R^{2n},\omega = \omega_0)$.

\begin{lemma}[{\cite[Corollary 3.5]{Milnor1973}}]
\label{symplectic modules are free}
Let $R$ be a Dedekind domain and $(M, \omega)$ a symplectic module over $R$ where $M$ is finitely generated projective. Then $M$ is a free module of even rank. In addition, any two symplectic forms on $R^{2n}$ are isomorphic.
\end{lemma}

The \emph{symplectic group} $\Sp{2n}{R}$ is the subgroup of $\on{GL}(R^{2n})$ given by all $g$ such that $\omega(g(x),g(y)) = \omega(x,y)$ for all $x,y\in R^{2n}$. Note that it follows immediately from the definitions that all symplectic bases of $R^{2n}$ lie in the same $\Sp{2n}{R}$-orbit.

Let $U \subseteq R^{2n}$ be a submodule. The \emph{orthogonal complement} $U^\perp \subseteq R^{2n}$ is the submodule consisting of all $x\in R^{2n}$ such that $\symf{x}{m}=0$ for all $m \in U$.
The submodule $U$ is \emph{isotropic} if $U\subseteq U^{\perp}$.

\subsubsection{Forms on $R^{2n}$ and $K^{2n}$}
Using that $K^{2n}=R^{2n}\otimes_{R}K$, a  symplectic form $\omega$ on $R^{2n}$ induces a symplectic form on $K^{2n}$ given by
\begin{align*}
    K^{2n}\times K^{2n}&\to K\\
    \left(r\otimes_{R} \frac{p}{q}, r'\otimes_{R} \frac{p'}{q'}\right)&\mapsto \frac{p}{q}\cdot\frac{p'}{q'}\symf{r}{r'}.
\end{align*}
By abuse of notation, we denote this form by $\omega$ as well. We also do not distinguish in notation between orthogonal complements with respect to the forms on $R^{2n}$ and $K^{2n}$.
In what follows, we verify that several notions translate well between  the symplectic modules $(R^{2n},\omega)$ and $(K^{2n},\omega)$.

\begin{lemma}
\label{lem_M_perp_intersections}
If $U \subseteq R^{2n}$ is a summand, then $(U\otimes_{R} K)^\perp \cap R^{2n} = U^\perp$.
\end{lemma}
\begin{proof}
First let $x \in (U\otimes_{R} K)^\perp \cap R^{2n}$. Then for all $y\in U\otimes_{R} K$, we have $\symf{x}{y}= 0$. But $U\subseteq U\otimes_{R} K$, so in particular, $\symf{x}{u} = 0$ for all $u \in U$. So as $x\in R^{2n}$, it is in $U^\perp \subseteq R^{2n}$.
Now let $x \in U^\perp$. By definition, $x \in R^{2n}$ and $\symf{x}{u} = 0$ for all $u\in U$. But then also $\symf{x}{u\otimes_{R} \frac{p}{q}} = \frac{p}{q}\symf{x}{u} = 0$ for all $u\in U$, $\frac{p}{q}\in K$. That is $x\in (U\otimes_{R} K)^\perp$, which proves the claim.
\end{proof}

\begin{lemma}
\label{lem_isotropic_summands_subspaces}
Let $U \subseteq R^{2n}$ be a summand. Then $U$ is isotropic if and only if $U\otimes_{R} K$ is. Hence, the isomorphism from \cref{iso btw summands and subspaces} restricts to a bijection between the rank-$r$ \emph{isotropic} summands of $R^{2n}$ and the dimension-$r$ \emph{isotropic} subspaces of $K^{2n}$.
\end{lemma}
\begin{proof}
If $U$ is isotropic, then $U\subseteq U^\perp$. Hence using \cref{lem_M_perp_intersections}, we get
\begin{align*}
	U\otimes_{R} K &\subseteq U^\perp \otimes_{R} K = \left((U\otimes_{R} K)^\perp \cap R^{2n}\right) \otimes_{R} K  = (U\otimes_{R} K)^\perp.
\end{align*}
Conversely, if $U\otimes_{R} K$ is isotropic, then $U\otimes_{R} K\subseteq (U\otimes_{R} K)^\perp$, so
\begin{align*}
	U = (U\otimes_{R} K) \cap R^{2n} \subseteq  (U\otimes_{R} K)^{\perp} \cap R^{2n} = U^\perp,
\end{align*}
where the last equality is again  \cref{lem_M_perp_intersections}.
\end{proof}

A consequence of \cref{lem_isotropic_summands_subspaces} is that \emph{every rank-1 submodule $I\subseteq R^{2n}$ is isotropic}:
As $I\otimes_{R} K$ is a dimension-1 subspace of $K^{2n}$, we have $I\otimes_{R} K= \text{span}(v)$ for some $v\in K^{2n}$. Since $\symf{v}{v}=0$ for any $v\in K^{2n}$, we have that  $\symf{I}{I}= 0$.

\begin{lemma}[{\cite[Lemma 3.11]{Scharlau1985}}]
\label{lem_orthogonal_complement_direct_summand}
Let $U \subseteq \Kn$ be a subspace of dimension $r$. Then $U^\perp \subseteq \Kn$ is a subspace of dimension $2n-r$.
\end{lemma}

\begin{definition}
We call a collection $\mathbf{L} = \ls L_1, \ldots, L_n, L_{-n}, \ldots, L_{-1}\rs$ of lines (i.e.~di\-men\-sion-1 subspaces) a \emph{symplectic frame of $K^{2n}$} if
\begin{equation*}
	L_i + L_j \text{ is isotropic if and only if } j\not=-i. 
\end{equation*}
\end{definition}

It is not hard to see that if $\mathbf{L}$ is a symplectic frame of $K^{2n}$, then $K^{2n} = L_1\oplus \ldots \oplus L_n \oplus L_{-n} \oplus \ldots \oplus L_{-1}$.

We call a frame $\mathbf{L}$ \emph{integral} if there is a symplectic basis $e_{1},\ldots, e_{n}, e_{-n},\ldots, e_{-1}$ of $R^{2n}$ such that for all $i$, the line $L_i = \langle e_{i}\rangle_K$ is the $K$-span of $e_i$.
As all symplectic bases of $R^{2n}$ lie in the same $\Sp{2n}{R}$-orbit, so do all integral symplectic frames.

\subsubsection{Finding subspaces with prescribed class group elements}

Below, we will need to find chains of subspaces of $K^{2n}$ that represent prescribed elements of the class group $\class(R)$. We now collect some technical lemmas for this.

The following is a special case of \cite[Proof of Proposition 5.4, Claim (1)]{Church2019}.
\begin{lemma}
\label{cfp-5.4-simplified}
    Let $M$ be a projective $R$-module with $\rk(M)\geq 2$ and let $c\in \class(R)$.
    Then there is a rank-$1$ summand $I\subset M$ such that $[I]=c$.
\end{lemma}

\begin{lemma}
\label{lem_relative_extension_of_summand}
Let $U\subset W$ be subspaces of $\Kn$ such that $U$ is isotropic, $W\subseteq U^\perp$ and $\dim(W) \geq \dim(U)+2$. Let $c\in \class(R)$. Then there is an isotropic subspace $V$ such that $\dim(V) = \dim(U)+1$, $[V] = c$ and
\begin{equation*}
	U\subset V \subset W.
\end{equation*}
\end{lemma}
\begin{proof}
We first prove the analogous result for summands of $\Rn$, i.e given summands $U_{0}\subset W_{0}$ of $\Rn$ such that $U_{0}$ is isotropic, $W_{0}\subseteq U_{0}^{\perp}$, and $\rk(W_{0})\geq \rk(U_{0})+2$, there is an isotropic summand $V_{0}$ such that $\rk(V_{0})=\rk(U_{0})+1$, $[V_{0}]=c$ and $U_{0}\subset V_{0} \subset W_{0}$.

As noted in \cref{sec_basics_rings_modules}, $U_{0}$ is a summand of $W_{0}$, so we can write $W_{0} = U_{0} \oplus U_{0}'$. As $\rk(W_{0}) \geq \rk(U_{0})+2$, we have $\rk(U_{0}')\geq 2$.
By \cref{cfp-5.4-simplified}, we can choose a rank-1 summand $I$ in $U_{0}'$ such that $[I]=c-[U_{0}]$. Set  
\begin{equation*}
	V_{0} \coloneqq U_{0}\oplus I.
\end{equation*}    
By definition, we have $U_{0}\subset V_{0} \subset W_{0}$ and $\rk(V_{0}) = \rk(U_{0})+1$. Both $U_{0}$ and $I$ are isotropic as $I$ is contained in $W_{0}\subseteq U_{0}^\perp$, this implies that $V_{0}$ is isotropic.
Furthermore, we have $[V_{0}]=[U_{0}]+ [I] = [U_{0}] + (c-[U_{0}]) = c$.

Now we prove the original result for vector spaces. Set $U_{0}\colonequals U\cap \Rn$ and $W_{0}\colonequals W\cap \Rn$. By \cref{lem_isotropic_summands_subspaces}, these are isotropic summands with $\rk(U_0) = \dim(U)$ and $\rk(W_0) = \dim(W)$. We have that $W_{0}$ is contained in $U_{0}^{\perp}$ since $U_{0}^{\perp}= U^{\perp}\cap \Rn$ by \cref{lem_M_perp_intersections}. By the previous paragraph, we can pick an isotropic summand $V_{0}$ such that $\rk(V_{0})=\rk(U_{0})+1$, $[V_{0}]=c$, and $U_{0}\subset V_{0} \subset W_{0}$. Let $V=V_{0}\otimes_{R}K$. We have that $\dim(V) = \rk(V_{0})=\dim(U)+1$ and $U\subset V \subset W$. We also have that $[V]=c$ since $[V]=[V\cap \Rn]=[V_{0}]=c$. 
\end{proof}

The next lemma is an example of an application of \cref{lem_relative_extension_of_summand} that we will need later on.

\begin{lemma}\label{chain of istropics}
Let $c_{1},\ldots, c_{n} \in \class(R)$. Then there exists a chain $U_{1}\subset \cdots \subset U_{n}$ of isotropic subspaces of $\Kn$ such that $\dim(U_{i})= i$ and $[U_{i}]=c_{i}$.
\end{lemma}
\begin{proof}
We first find $U_1\subset \Kn$ by \cref{cfp-5.4-simplified} and \cref{lem_isotropic_summands_subspaces}.
Now assume by induction that for some $i < n$, we have constructed $U_{1}\subset \cdots \subset U_{i}$ with the desired properties. By \cref{lem_orthogonal_complement_direct_summand}, $U_{i}^{\perp}$ is a subspace of dimension $2n-i \geq i+2$.
Hence applying \cref{lem_relative_extension_of_summand} to the pair $U_{i}\subset U_{i}^{\perp}$, we find an isotropic subspace $U_{i+1}$ such that $\dim(U_{i+1}) = i+1$, $[U_{i+1}] = c_{i+1}$ and $U_{i}\subset U_{i+1}$.
\end{proof}

\subsection{The group of signed permutations}
\label{sec_signed_perm_group}
Let
\begin{equation*}
	[n] = \ls 1, \ldots, n \rs \text{ and } [\pm n] = \ls 1, \ldots, n, -n, \ldots, -1 \rs.
\end{equation*}
We denote by $\signp$ the group of signed permutations of rank $n$. That is, $\signp$ is the group of all permutations $\sigma$ of $[\pm n]$ that satisfy
\begin{equation*}
	\sigma(-i) = -\sigma(i).
\end{equation*}
We multiply the elements of $\signp$ as functions, i.e. 
\begin{equation*}
	(\sigma \sigma')(i) = \sigma (\sigma'(i)).
\end{equation*}
The group $\signp$ is the finite Coxeter group of type $\mathtt{C}_n = \mathtt{B}_n$. It has a generating set $S$ of simple reflections $s_1, \ldots, s_n$ given by
\begin{align*}
	s_i &= [1, \ldots, i-1, i+1,i, i+2, \ldots, n] \text{ for } i <n, \\
	s_n &= [1, \ldots, n-1, -n],
\end{align*}
i.e.~the $n-1$ adjacent transpositions and the element exchanging $n$ and $-n$. (Here, we use the ``window'' notation, see \cite[Section 8.1]{BB:CombinatoricsCoxetergroups}.)
We write 
\begin{equation*}
	\len(\sigma) = \len_S(\sigma)
\end{equation*}
for the word length of $\sigma$ with respect to the generating set $S = \ls s_1, \ldots, s_n \rs$.

\subsection{Symplectic buildings}
\label{sec_buildings}
(Spherical) buildings are a class of simplicial complexes that have a rich structure theory. For an introduction and general reference, see \cite{AB:Buildings}.
In this work, we are only concerned with one specific type of building:
Let $\buildingssymp(K^{2n})$ denote the poset of non-trivial isotropic subspaces of $K^{2n}$. It comes with a natural action of $\Sp{2n}{K}$ and its subgroup $\Sp{2n}{R}$.
The order complex (or realization) of this poset, which we also write as $\buildingssymp(K^{2n})$, is an $(n-1)$-dimensional simplicial complex, the \emph{spherical building of type $\mathtt{C}_n $ over $K$}. 

The building $\buildingssymp(K^{2n})$  has a system of distinguished subcomplexes, so-called \emph{apartments}:
Let $\mathbf{L} = \ls L_1, \ldots, L_n, L_{-n}, \ldots, L_{-1}\rs$ be a symplectic frame of $K^{2n}$. Note that if $\sigma\in \signp$ is a signed permutation and $m \leq n$, then the direct sum
\begin{equation*}
	L_{\sigma([m])}\coloneqq\bigoplus_{i\in \sigma([m])} L_i
\end{equation*}
is an isotropic subspace.
We let $\apartment_{\mathbf{L}}$ be the subcomplex of $\buildingssymp(K^{2n})$ spanned by all subspaces of this form. This subcomplex is called an \emph{apartment}.
Its maximal simplices (the ``chambers'') are given by chains of the form
\begin{equation}
\label{eq_L_sigma}
	\mathbf{L}_{\sigma}\coloneqq  L_{\sigma([1])} \subset \ldots \subset L_{\sigma([n])}.
\end{equation}
It is isomorphic to the barycentric subdivision of the boundary of an $n$-dimensional cross-polytope (the Coxeter complex of $\signp$, see the left of \cref{fig_motivation}). In particular, its geometric realization is homeomorphic to an $(n-1)$-sphere. The inclusion $\apartment_{\mathbf{L}} \hookrightarrow \buildingssymp(K^{2n})$ sends its fundamental class to a non-trivial \emph{apartment class}\footnote{This class is only determined up to a choice of sign for the fundamental class. This ambiguity is however not relevant for our considerations, so we will neglect it from now on.} given by
\begin{equation*}
	[\apartment_{\mathbf{L}}] \coloneqq \sum_{\sigma\in\signp} (-1)^{\len(\sigma)}\mathbf{L}_{\sigma} \in \tilde{H}_{n-1}(\buildingssymp(K^{2n})).
\end{equation*}
The apartment $\apartment_{\mathbf{L}}$ and its class $[\apartment_{\mathbf{L}}]$ are called \emph{integral} if the frame $\mathbf{L}$ is integral. As all integral frames lie in the same $\Sp{2n}{R}$-orbit, so do all integral apartments.
\newline

By the Solomon--Tits Theorem \cite{Sol:Steinbergcharacterfinite}, the geometric realization of $\buildingssymp(K^{2n})$ is homotopy equivalent to a wedge of $(n-1)$-spheres and its only non-trivial reduced homology group, 
which we denote by 
\begin{equation*}
	\Stsymp = \Stsymp_n(K) \coloneqq \tilde{H}_{n-1}(\buildingssymp(K^{2n})),
\end{equation*}
is generated by apartment classes. We call $\Stsymp$ the \emph{(symplectic) Steinberg module}.

Borel--Serre \cite{BS:Cornersarithmeticgroups} showed that for number fields $K$, the Steinberg module is closely related to the (co-)homology of the arithmetic group $\Sp{2n}{\mcO_K}$. The following is a special case of their result.
\begin{theorem}[Borel--Serre duality]
\label{thm_Borel_Serre}
Let $K$ be a number field and $R = \mcO_K$ its ring of integers. 
For all $i$, there is an isomorphism
\begin{equation*}
	H^{\nu_{n} -i}( \Sp{2n}{R};\mbQ) \cong H_i(\Sp{2n}{R}; \Stsymp \otimes_{\mbZ} \mbQ),
\end{equation*}
where $\nu_{n}$ is the virtual cohomological dimension of $\Sp{2n}{R}$.
\end{theorem}

\section{Surjectivity on apartment classes}
\label{sec_surjectivity_symplectic}
\cref{thm_Borel_Serre} in particular implies that $H^{\nu_{n}}(\Sp{2n}{R}; \mbQ) \cong(\Stsymp\otimes_{\mbZ} \mbQ)_{\Sp{2n}{R}}$. Hence, showing that $H^{\nu_{n}}(\Sp{2n}{R}; \mbQ)$ is non-trivial (which is our goal for proving \cref{thm_non_trivial_cohomology}) is equivalent to showing that the coinvariants $(\Stsymp\otimes_{\mbZ} \mbQ)_{\Sp{2n}{R}}$ do not vanish. In order to do this, we will find a surjection from $(\Stsymp\otimes_{\mbZ} \mbQ)_{\Sp{2n}{R}}$ to a vector space of sufficiently high dimension.
We construct this surjection on the level of simplicial complexes in this section.

\subsection{The complex $X_n$ and the map $\psi$}
\label{sec_defn_X_n}
The poset $X_n(\class(R))$ (see \cite[Example 5.3]{Church2019}) has underlying set $[n]\times \class(R)$ and the order relation is given by $(i, a)<(j, b)$ if and only if $i<j$.

The order complex of this poset, which we also denote by $X_n(\class(R))$, is an iterated join of $n$ copies of the discrete set $\class(R)$. By using that $\hcl=\bigvee_{|\hcl|-1} S^{0}$, the geometric realization $|X_n(\class(R))|$ is homotopy equivalent to a wedge of $(|\class(R)|-1)^n$ many $(n-1)$-spheres:
\begin{align*}
  |X_n(\class(R))|\simeq \ast_{i=1}^{n}\bigvee_{|\hcl|-1} S^{0}&\simeq\bigvee_{(|\hcl|-1)^{n}} S^{0}\ast \cdots \ast S^{0} \hspace{3mm} (\text{n times})\\
  & \simeq \bigvee_{(|\hcl|-1)^{n}} S^{n-1}.
\end{align*}
In particular,
we have that 
$\tilde{H}_{n-1}(|X_n(\class(R))|)\cong \big(\tilde{H}_{0}(\bigvee_{|\hcl|-1} S^{0})\big)^{\otimes n}.$
We have the following description of $\tilde{H}_{n-1}(|X_n(\class(R))|)$ in terms of $n$-tuples of pairs of distinct elements of $\hcl$.

For each maximal chain $(1, a_1 )<(2, a_2 )<\ldots<(n, a_n )$ of $X_n(\class(R))$, there is a corresponding maximal simplex in the order complex, which we denote by $(a_1 ,\ldots, a_n )$.
Let $\mathbf{S} = (S_1, \ldots, S_n)$ be an $n$-tuple
of pairs $S_i = (a_i , b_i)\in \class(R)^{2}$, with $ a_i \neq b_i$.

For $\mathbf{e} = (\epsilon_1, \ldots, \epsilon_n)\in \ls 0,1 \rs^n$, we set 
\begin{equation*}
	c_i^{\mathbf{e}} = 
\begin{cases}
b_i & \epsilon_i = 0 \\
a_i & \epsilon_i = 1.
\end{cases}
\end{equation*}
We define $\apartment_{\mathbf{S}}$ to be the subcomplex of $X_n(\class(R))$ spanned by the $2^n$ simplices $C_{\mathbf{e}} = (c_1^{\mathbf{e}}, \ldots, c_n^{\mathbf{e}})$. The geometric realization of this subcomplex is homeomorphic to an $(n-1)$-sphere, given as the join of $n$ copies of $S^0$. Its fundamental class in $\tilde{H}_{n-1}(X_n(\class(R))$ is given by
\begin{equation*}
	[\apartment_{\mathbf{S}}] \coloneqq \sum_{\mathbf{e}\in \ls 0,1 \rs^n} (-1)^{\sum \epsilon_i}[C_{\mathbf{e}}].
\end{equation*}
The homology group $\tilde{H}_{n-1}(X_n(\class(R))$ is generated by the classes $[\apartment_{\mathbf{S}}]$, where $\mathbf{S}$ ranges over the set of n-tuples of pairs of distinct elements of $\class(R)$.

There is a poset map $\psi$ defined by 
\begin{align*}
	\psi\colon \buildingssymp(\Kn) &\to X_{n}(\class(R)) \\
	U &\mapsto (\dim(U),[U]).
\end{align*}
It is invariant under the action of $\Sp{2n}{R}$ on $\buildingssymp(\Kn)$, i.e.~for $g\in \Sp{2n}{R}$, we have $\psi(g(U))= \psi(U)$. Using \cref{chain of istropics}, one can see that $\psi$ is surjective.
The goal of this section is to show that it also induces a surjective map on homology:

\begin{proposition}
\label{prop_homology_surjectivity}
The induced map 
\begin{equation*}
	\psi_{*}\colon \Stsymp = \tilde{H}_{n-1}(\buildingssymp(\Kn))\to \tilde{H}_{n-1}(X_{n}(\hcl))\cong \mbZ^{(|\class(R)|-1)^n}
\end{equation*}
is surjective.
\end{proposition}

\label{sec_beginning_of_proof_hom_surj}
As $\tilde{H}_{n-1}(X_{n}(\hcl))$ is generated by the classes $[\apartment_{\mathbf{S}}]$, it is sufficient to find a preimage for each such class. We now fix an $n$-tuple $\mathbf{S} = (S_1, \ldots, S_n)$, where $S_i = (a_i , b_i)\in \class(R)^{2}$, with $ a_i \neq b_i$.
Our aim is to construct a symplectic frame $\mathbf{L}$ of $\Kn$ such that $\psi_{*}$ maps the apartment class $[\apartment_{\mathbf{L}}]$ to $[\apartment_{\mathbf{S}}]$.

\subsection{Proof of \cref{prop_homology_surjectivity}, Step 1: Constructing $\mathbf{L}$}
\label{sec_psi_induces_surjectivity}

\subsubsection{Motivation for $n=3$}
\label{sec_motivation}
\begin{figure}
\begin{center}
\includegraphics{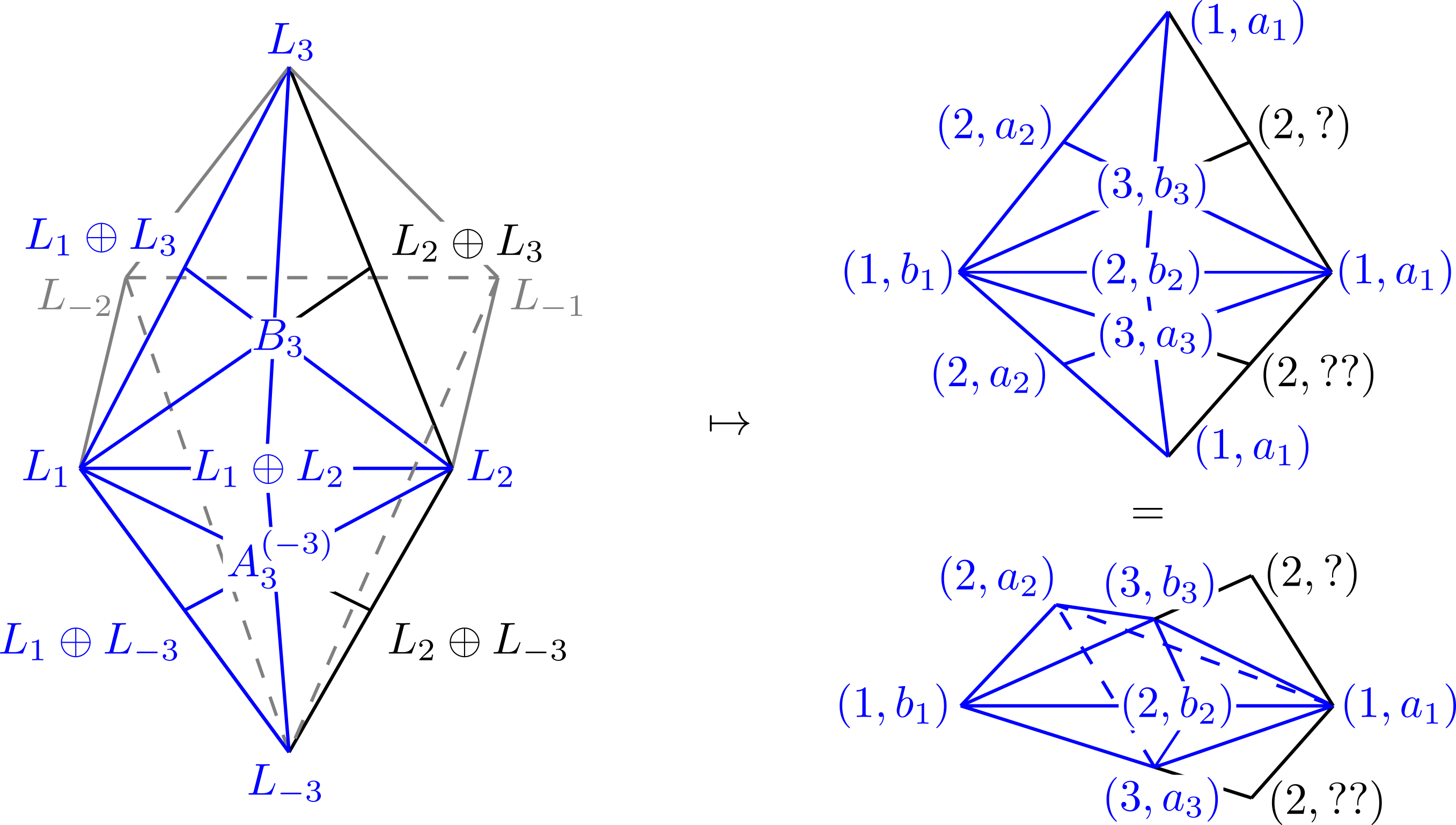}
\end{center} 
\caption{An apartment $\apartment_\mathbf{L}$ in $\buildingssymp(K^{6})$ such that $\psi_*([\apartment_\mathbf{L}]) = [\apartment_\mathbf{S}]$. The left shows $\apartment_\mathbf{L}$. Blue and black highlight $\st(B_3) \cup \st(A^{(-3)}_{3})$, where $B_3=L_{1}\oplus L_{2}\oplus L_{3}$ and $A^{(-3)}_{3}=L_{1}\oplus L_{2}\oplus L_{-3}$; the simplicial structure in the remaining grey part is omitted.
The top right shows the images of these simplices without identifying faces that are the same. The bottom right shows the subcomplex of $X_{n}(\hcl)$ after making these identifications. It is given by the 2-sphere $\apartment_\mathbf{S}$ (in blue) together with two triangular ``fins'' that are homologically trivial (in black).}
\label{fig_motivation}
\end{figure}

Before giving the actual construction of the symplectic frame $\mathbf{L}$ in \cref{sec_definition_L}, we explain the idea for the case $n=3$, see also \cref{fig_motivation}.

Each chamber of $\apartment_{\mathbf{L}}$ belongs to the star of $L_{\pm 1}\oplus L_{\pm 2}\oplus L_{\pm 3}$ for some unique choice of signs of the indices.
Our strategy is to construct $\mathbf{L}$ in such a way that
\begin{enumerate}
    \item\label{map-to-A_{S}} $\psi_{*}([\apartment_{\mathbf{L}}]|_{\text{star}(L_{1} \oplus L_{2} \oplus L_{3}})+[\apartment_{\mathbf{L}}]|_{\text{star}(L_{1} \oplus L_{2} \oplus L_{-3}}))=[\apartment_{\mathbf{S}}]$ and
    \item\label{other-stars-cancel} $\psi_{*}([\apartment_{\mathbf{L}}]|_{\text{star}(L_{j} \oplus L_{k} \oplus L_{l}}))=0$ 
    for all other $j=\pm 1, k=\pm 2, l=\pm 3$.
\end{enumerate}

To obtain \cref{map-to-A_{S}}, we first construct isotropic subspaces $B_{3}, \, A^{(-3)}_{3} \subset K^{6}$ of maximal dimension with $[B_{3}]=b_{3}$ and $[A^{(-3)}_{3}]=a_{3}$. We then define a ``partial symplectic frame'' $\mathbf{L}'=\{L_{1}, L_{2}, L_{3}, L_{-3} \}$ such that $B_{3}=L_{1}\oplus L_{2}\oplus L_{3}$, $A^{(-3)}_{3}=L_{1} \oplus L_{2} \oplus L_{-3}$, $[L_{1}]=b_{1}$ and  $[L_{j}]=a_{1}$ for $j\neq 1$.
This is done in such a way that we can pick $[L_{j}\oplus L_{k}]$ however we want for ``sufficiently many'' possible combinations of $L_{j}, L_{k}\in \mathbf{L}'$ (namely the ones highlighted in blue in \cref{fig_motivation}). The procedure involves repeatedly applying \cref{lem_relative_extension_of_summand} and \cref{chain of istropics}, which allow us to construct nested isotropic subspaces with prescribed class group elements. 
If we set $[L_{1}\oplus L_{2}]$ to be $b_{2}$ and $[L_{1}\oplus L_{3}]$ and $[L_{1} \oplus L_{-3}]$ to be $a_2$, then as shown in \cref{fig_motivation}, $\psi_{*}$ sends the chambers arising from this partial symplectic frame to $[\apartment_{\mathbf{S}}]$. 

Completing $\mathbf{L'}$ to a symplectic frame that satisfies \cref{other-stars-cancel} is more delicate. 
The main difficulty is that $\mathbf{L'} = \ls L_1, L_2, L_3, L_{-1} \rs$ puts restrictions on possible values of $L_{-2}$ and $L_{-1}$. This makes it harder to control the class group elements of the remaining vertices of $\apartment_\mathbf{L}$. While we can assure that $[L_{-2}]=[L_{-1}]=a_1$, we cannot control the class group elements of most of the other vertices. What we can assure though is that the chambers with such vertices (which we call ``bad'' below) occur in pairs that cancel each other out in homology. This is similar to the situation of the four black chambers in \cref{fig_motivation} and explained in \cref{sec_image_bad_simplices}.

\subsubsection{Defining $\mathbf{L}$}
\label{sec_definition_L}
We first define $L_1, \ldots, L_n$ and then $L_{-n}, \ldots, L_{-1}$.
\begin{lemma}
We may choose isotropic subspaces $B_{i}$ for $1\leq i\leq n$ and $A^{(j)}_{i}$ for $1\leq i< j\leq n$ such that
    \begin{align*}
        0=B_{0}\subset B_{1}\subset \cdots\subset B_{n}, &\quad \dim(B_{i})=i, [B_{i}]=b_{i},\\
        A^{(j)}_{1}\subset A^{(j)}_{2}\subset \cdots \subset A^{(j)}_{j-1} \subset B_{j}, &\quad \dim(A^{(j)}_{i})=i,\, [A^{(j)}_{i}]=a_{i}\text{ and } B_{i-1}\subset A^{(j)}_{i}
        .
    \end{align*}
\end{lemma}
\begin{proof}
We find the $B_i$ by \cref{chain of istropics}.
Given $0=B_{0}\subset B_{1} \subset \cdots\subset B_{j+1}$, we construct $A^{(j+1)}_{1}\subset A^{(j+1)}_{2}\subset \cdots \subset A^{(j+1)}_{i}$ as follows:

Applying \cref{lem_relative_extension_of_summand} to the pair $B_{j-1}\subset B_{j+1}$, we find an isotropic subspace $A^{(j+1)}_{j}$ of dimension $j$ such that $[A^{(j+1)}_{j}]=a_j$ and $B_{j-1}\subset A^{(j+1)}_{j} \subset B_{j+1}$.
Now assume that for some $1\leq i < j$, we have constructed $A^{(j+1)}_{i+1}\subset \cdots \subset A^{(j+1)}_{j}$. Then applying \cref{lem_relative_extension_of_summand} to the pair $B_{i-1} \subset A^{(j+1)}_{i+1}$, we obtain the isotropic subspace $A^{(j+1)}_{i}$ of dimension $i$ such that $[A^{(j+1)}_{i}]=a_i$ and $B_{i-1}\subset A^{(j+1)}_{i} \subset  A^{(j+1)}_{i+1}$.
\end{proof}

We define $L_1,\ldots, L_n$ by setting $
L_{1}\colonequals B_{1}$ and $L_{j}\colonequals A^{(j)}_{1}\text{ for }2\leq j\leq n$. We will show in \cref{sat of more frames} that $\{L_{1},\ldots, L_{i}\}$ is a frame of $B_{i}$ for all $i=1,\ldots, n$.

We now inductively construct $ L_{-n}, L_{-(n-1)},\ldots, L_{-1}$. Let $1 \leq j \leq n$ and assume that by induction, we have constructed lines $ L_{-n}, L_{-(n-1)},\ldots, L_{-(j+1)}$ such that 
\begin{equation}
\label{eq_Bk-1_in_L_j}
	B_{j-1}\subseteq L_{-k}^\perp \text{ for all } j+1\leq k\leq n.
\end{equation}
Observe that this condition is void if $j=n$, which is the induction beginning.

We want to obtain $L_{-j}$ from this. Let $W_{j+1}$ be the subspace
\begin{equation*}
	W_{j+1}\colonequals B_{j-1}^\perp \cap(\bigcap_{k=j+1,\ldots ,n}L^{\perp}_{k})\cap (\bigcap_{k=j+1,\ldots, n}L^{\perp}_{-k}).
\end{equation*}

\begin{lemma}
We may choose isotropic subspaces $A^{(-j)}_{i}$ for $1\leq i\leq j$ such that
    \begin{gather*}
 A^{(-j)}_{1}\subset A^{(-j)}_{2}\subset \cdots \subset A^{(-j)}_{j}\subset W_{j+1}, \\ \dim(A^{(-i)}_{i})=i,\, [A^{(-j)}_{i}]=a_{i}\text{ and } B_{i-1}\subset A^{(-j)}_{i}.
    \end{gather*}
\end{lemma}
\begin{proof}
By \cref{lem_orthogonal_complement_direct_summand}, $B_{j-1}^\perp$ has dimension $2n-j+1$ and $L^{\perp}_{\pm i}$ has dimension $2n-1$. Hence  $W_{j+1}$ has dimension at least $(2n-j+1)-2(n-j)=j+1$.
Both $B_{j-1}$ and $L_{k}$, for $k\in \ls j+1, \ldots, n \rs$ are contained in the isotropic subspace $B_n$, so $B_{j-1}\subseteq L_{k}^\perp$ for all such $k$. By assumption (\cref{eq_Bk-1_in_L_j}), we also have $B_{j-1}\subseteq L_{-k}^\perp$, so $B_{j-1}$ is contained in $W_{j+1}$.
Clearly, $W_{j+1}\subseteq B_{j-1}^\perp$. Hence, we can  apply \cref{lem_relative_extension_of_summand} to the pair $B_{j-1}\subset W_{j+1}$ to find an isotropic subspace $A^{(-j)}_{j}$ of dimension $j$ such that $[A^{(-j)}_{j}]=a_j$ and $B_{j-1}\subset A^{(-j)}_{j} \subset W_{j+1}$.
Now assume that for some $1\leq i <j$, we have constructed $A^{(-j)}_{i+1}\subset \cdots \subset A^{(-j)}_{j}$ with the desired properties. Then applying \cref{lem_relative_extension_of_summand} to the pair $B_{i-1} \subset A^{(-j)}_{i+1}$, we obtain the isotropic subspace $A^{(-j)}_{i}$ of dimension $i$ such that $[A^{(-j)}_{i}]=a_i$ and $B_{i-1}\subset A^{(-j)}_{i} \subset A^{(-j)}_{i+1}$.
\end{proof}

We define $L_{-j} \coloneqq A^{(-j)}_{1}$.

As $L_{-j}\subseteq W_{j+1} \subseteq B_{j-1}^\perp \subseteq B_{j-2}^\perp$, \cref{eq_Bk-1_in_L_j} also holds with $j+1$ replaced by $j$, i.e.
\begin{equation*}
	B_{j-2} \subseteq L_{-k}^\perp \text{ for all } j\leq k\leq n.
\end{equation*}
Hence we can continue and inductively define $L_{-n}, \ldots, L_{-1}$.

\subsubsection{Properties of $\mathbf{L}$}
Set $\mathbf{L} \coloneqq \ls L_1, \ldots, L_n, L_{-n}, \ldots, L_{-1}\rs$. We now show that $\mathbf{L}$ is indeed a symplectic frame of $\Kn$ and describe some of the vertices that are contained in the corresponding apartment. We begin with two auxiliary lemmas.

\begin{lemma}
\label{lem_intersection_A_k_B_k}
For all $j>0$, we have that $B_{i}\cap A^{(j)}_{i}= B_{i-1}$ for all $i <j$, and $B_{i}\cap A^{(-j)}_{i}= B_{i-1}$ for all $i \leq j$.
\end{lemma}
\begin{proof}
The inclusion $B_{i}\cap A^{(\pm j)}_{i} \supseteq B_{i-1}$ is clear. On the other hand, $B_{i}\cap A^{(\pm j)}_{i}$ is a subspace of $B_i$ and contains $B_{i-1}$, so it must have dimension $i-1$ or $i$. If it had dimension $i$, this would imply $B_{i}= A^{(\pm j)}_{i}$, which is impossible since $[B_{i}]=b_i\neq a_i= [A^{(\pm j)}_{i}]$. Hence, $\dim(B_{i}\cap A^{(\pm j)}_{i}) = i-1$. This implies the equality.
\end{proof}

\begin{lemma}
\label{L_not_in_B}
$L_j$ is not contained in $B_{j-1}$ and $L_{-j}$ is not contained in $B_j$.
\end{lemma}
\begin{proof}
We show that $L_{j}$ is not contained in $B_{i}$ for any $0\leq i < j$ and $L_{-j}$ is not contained in $B_{i}$ for any $0\leq i \leq j$. For $i=0$,
this is clear as $B_0 = \ls 0 \rs$.

Now assume that $i$ is in the corresponding range and $L_{\pm j}$ is not contained in $B_{i-1}$. Then $L_{\pm j}$ also cannot be contained in $B_{i}$: Suppose that it is. Then, since $L_{\pm j}=A^{(\pm j)}_{1}\subseteq A^{(\pm j)}_{i}$, we have $	L_{\pm j}\subseteq B_{i}\cap A^{(\pm j)}_{i} = B_{i-1}$,
where we used \cref{lem_intersection_A_k_B_k}. This is a contradiction.
\end{proof}

The following lemma describes some vertices of the apartment $\apartment_{\mathbf{L}}$.

\begin{lemma}\label{sat of some frames}
\label{sat of more frames}
Let $i\in \ls 1, \ldots, n \rs$.
\begin{enumerate}
\item \label{it_sat_Bi}$L_1 + L_2 +\ldots + L_i= B_i$.
\item \label{it_sat_Aij} $B_i + L_{j} = A^{(j)}_{i+1}$ for all $j \in \ls -(i+1), \pm (i+2), \ldots, \pm n \rs$.
\end{enumerate}
\end{lemma}
\begin{proof}
We prove \cref{it_sat_Bi} by induction. The base case $i=1$ is trivial because $L_{1}=B_{1}$. Suppose the claim holds for $i-1$. We have that $L_{i}$ and $B_{i-1}$ are contained in $B_{i}$ by construction. Since $B_{i-1}=L_1 + L_2 +\ldots + L_{i-1}$ has dimension $i-1$, to prove the claim for $i$, it suffices to note that by \cref{L_not_in_B}, $L_{i}$ is not contained in $B_{i-1}$.
 
For \cref{it_sat_Aij}, note that $B_{i}\subset A^{(j)}_{i+1}$. Since $L_j$ is also in $A^{(j)}_{i+1}$ and the dimension of $B_i$ is $i$, either $B_i +  L_j  =A^{(j)}_{i+1}$ or $L_j$ is contained in $B_i$. The latter is impossible because of \cref{L_not_in_B}.
\end{proof}

Using the above, we can show that $\mathbf{L}$ is indeed a frame.

\begin{lemma}
\label{lem_orthogonality_L_negative}
$\mathbf{L}$ is a symplectic frame of $\Kn$.
\end{lemma}
\begin{proof}
We have to show that $L_j + L_k$ is isotropic if and only if $k\not=-j$. We will assume throughout that $k>0$ and $j>0$, and show that $L_{j}+ L_{k}$, $L_{-j}+ L_{-k}$, and $L_{-j}+ L_{k}$ (for $j\not= k$) are all isotropic.  For the first case ($L_{j}+ L_{k}$ is isotropic),  $L_1, \ldots, L_n$ are all contained in $B_n$. Since $B_n$ is isotropic, this implies that they are orthogonal to each other. For the second case ($L_{-j}+ L_{-k}$), without loss of generality, we can assume $j<k$. Since $L_{-j}\subseteq W_{j+1}\subseteq L_{-k}^\perp$, we have that $L_{-j}+ L_{-k}$ is isotropic.

For the third case ($L_{-j}+ L_{k}$), consider the subcase that $j<k$. Since $L_{-j}\subseteq W_{j+1}\subseteq L_{k}^\perp$, we have that $L_{-j}+ L_{-k}$ is isotropic for $j<k$. Consider the subcase $k<j$. Since $L_{-j}\subseteq W_{j+1}\subseteq B_{j-1}^\perp$ and $L_{k}\subseteq B_{j-1}$ for $k<j$, we have that $L_{-j}+ L_{-k}$ is isotropic for $k<j$.

Hence, we just need to show that $L_{j}+L_{-j}$ is not isotropic. For $1\leq j\leq n$, define
\begin{equation*}
	B_{2n-j+1} \coloneqq L_{1}+ \cdots + L_{n} + L_{-n}+ \cdots + L_{-j}.
\end{equation*}
We will inductively prove the following three things:
\begin{enumerate}
\item \label{it_partial_fram_symplectic_relations}$L_{j} + L_{-j}$ is not isotropic.
\item \label{it_partial_frame_rank}$\dim(B_{2n-j+1}) = 2n-j+1$.
\item \label{it_Bi_in_orthogonal_complements} $B_{2n-j+1}^\perp=B_{j-1}$.

\end{enumerate}

\paragraph{Claim 1:} If \cref{it_partial_frame_rank} holds for some $j$, then so does \cref{it_Bi_in_orthogonal_complements}.

By \cref{sat of some frames}, $B_{j-1} = L_1 + \ldots + L_{j-1}$. 
As $L_1, \ldots, L_{j-1}$ are orthogonal to $L_1, \ldots, L_n$ and $ L_{-n}, \ldots, L_{-j}$, we have $
	B_{2n-j+1}^\perp \supseteq B_{j-1}$.
Hence, it suffices to observe that $B_{2n-j+1}^{\perp}$ and $B_{j-1}$ have the same dimension:
using \cref{lem_orthogonal_complement_direct_summand} and \cref{it_partial_frame_rank}, we have
\begin{equation*}
	\dim(B_{2n-j+1}^{\perp}) = 2n-\dim(B_{2n-j+1}) = 2n-(2n-j+1) = j-1.
\end{equation*}
This implies Claim 1.

Now assume that \cref{it_partial_fram_symplectic_relations}, \cref{it_partial_frame_rank} and \cref{it_Bi_in_orthogonal_complements}
hold for $j+1$. 
In the induction beginning $j=n$, this is the case: \cref{it_partial_fram_symplectic_relations} is void because for $j=n$, there is no $L_{j+1}= L_{n+1}$;
\cref{it_partial_frame_rank} is satisfied because $L_{1}+\cdots + L_{n} = B_n$ by \cref{sat of some frames}; \cref{it_Bi_in_orthogonal_complements} follows from that by Claim 1.

We first show that \cref{it_partial_fram_symplectic_relations} holds for $j$, i.e.~$L_{j}+L_{-j}$ is not isotropic. Suppose that it was isotropic.
Then
\begin{align*}
	L_{-j}\subseteq B_{2n-(j+1)+1}^{\perp} = B_j,
\end{align*}
where the equality is \cref{it_Bi_in_orthogonal_complements} for $j+1$. This is a contradiction to \cref{L_not_in_B}.

From this, it follows that \cref{it_partial_frame_rank} holds for $j$, i.e. $\dim( B_{2n-j+1}) = 2n-j+1$:
By assumption, \cref{it_partial_frame_rank} holds for $j+1$, i.e. $\dim( B_{2n-(j+1)+1}) = 2n-j$.
Furthermore, $B_{2n-(j+1)+1}\subseteq L_{j}^\perp$, so
\begin{equation*}
	L_{-j} \not\subset B_{2n-(j+1)+1}
\end{equation*}
because \cref{it_partial_fram_symplectic_relations} holds for $j$, so $L_{-j}\not\subseteq L_{j}^\perp$.
\end{proof}

\subsection{Proof of \cref{prop_homology_surjectivity}, Step 2: Good and bad elements of $\signp$}
Let $S = \ls s_1, \ldots, s_n \rs$ be the set of simple reflections in $\signp$  that was defined in 
\cref{sec_signed_perm_group}.

\begin{definition}
\label{def_good_bad_elements}
We say that a permutation $\sigma\in \signp$ is \emph{good} if it is a subword of the Coxeter word $s_n\cdots s_1$, i.e.~if it can be written in the form
\begin{equation*}
	\sigma = s_n^{\epsilon_n}\cdots s_1^{\epsilon_1}
\end{equation*}
for some $\epsilon_i \in \ls 0,1 \rs$.
If this is not the case, we say that $\sigma$ is \emph{bad}.
\end{definition}

\begin{lemma}
\label{lem_descrip_of_good_elements}
For $\sigma \in \signp$, the following are equivalent:
\begin{enumerate}
\item \label{it_sigma_good} $\sigma$ is good.
\item \label{it_sigma_on_subsets} For all $2\leq i \leq n$, we have $[i-1]\subset \sigma([i])$.
\end{enumerate}
Furthermore, if these hold, $\sigma = s_n^{\epsilon_n}\cdots s_1^{\epsilon_1}$, then $\epsilon_i = 0$ if and only if $\sigma([i]) = [i]$.
\end{lemma}
\begin{proof}
It is not hard to see that \cref{it_sigma_good} implies \cref{it_sigma_on_subsets}.

We prove that \cref{it_sigma_on_subsets} implies \cref{it_sigma_good} together with the description of $\epsilon_i$ by induction on $n$. If $n=1$, this is easy to check: $\signp[1]\cong \mbZ/2\mbZ$ only contains $s_1^0 = \on{id}$ and the element $s_1^1 = s_1$ that swaps $1$ with $-1$.

So assume that $n\geq 2$ and that the claim is true for $n-1$.
Let $\sigma$ satisfy \cref{it_sigma_on_subsets}.  
Then $\ls 1 \rs = [1]\subset \sigma([2])$, i.e.~$\sigma^{-1}(1)\in \ls 1, 2 \rs$.

If $\sigma^{-1}(1)=1$, then $\sigma$ restricts to an element in the group of signed permutations of the set $\ls 2, \ldots, n, -2, \ldots -n \rs$. This group is generated by the simple reflections $s_2, \ldots, s_n$ and is isomorphic to $\signp[n-1]$. Hence by induction, we have
\begin{equation*}
	\sigma = s_n^{\epsilon_n}\cdots s_2^{\epsilon_2} =  s_n^{\epsilon_n}\cdots s_2^{\epsilon_2}s_1^0,
\end{equation*}
where $\epsilon_i = 0 $ if and only if $\sigma(\ls 2, \ldots, i \rs) = \ls 2, \ldots, i\rs$. As we assumed $\sigma(1) = 1$, the latter is equivalent to saying that $\sigma([i]) = [i]$.

If $\sigma^{-1}(1)=2$, then $\sigma' \coloneqq  \sigma s_1$ still satisfies \cref{it_sigma_on_subsets} because $\sigma'([i]) = \sigma[i]$ for all $i \geq 2$ (this follows because $s_1([i]) = [i]$ for all $i\geq 2$). However, we also have $\sigma'^{-1}(1)=1$. Hence by the paragraph above,
\begin{equation*}
	\sigma' = s_n^{\epsilon_n}\cdots s_2^{\epsilon_2},
\end{equation*}
where $\epsilon_i = 0 $ if and only if $\sigma([i]) = [i]$,
i.e.
\begin{equation*}
\sigma = \sigma' s_1 = s_n^{\epsilon_n}\cdots s_2^{\epsilon_2} s_1^1
\end{equation*}
satisfies the claim.
\end{proof}

In particular, \cref{lem_descrip_of_good_elements} implies that the sequence $\mathbf{e} = \mathbf{e}_{\sigma}=(\epsilon_{1},\ldots,\epsilon_{n})$ is an invariant of a good permutation $\sigma$, so it does not depend on the word in $S$
we choose to represent $\sigma$.
It gives a 1-to-1 correspondence between such sequences (or equivalently elements in $\ls 0,1 \rs^n)$ and good elements in $\signp$.

Recall that the apartment class $[\apartment_{\mathbf{L}}]$ is given by
\begin{equation*}
	[\apartment_{\mathbf{L}}] \coloneqq \sum_{\sigma\in\signp} (-1)^{\len(\sigma)}\mathbf{L}_{\sigma}
\end{equation*}
with $\mathbf{L}_{\sigma}$ as in \cref{eq_L_sigma}.
    We define
\begin{equation*}
    [\apartment_{\mathbf{L}}]^{\on{good}}\colonequals \displaystyle\sum_{\sigma\in\signp \text{ good}} (-1)^{\len(\sigma)}\mathbf{L}_{\sigma} \text{ and } [\apartment_{\mathbf{L}}]^{\on{bad}}\colonequals \displaystyle\sum_{\sigma\in\signp \text{ bad}} (-1)^{\len(\sigma)}\mathbf{L}_{\sigma},
\end{equation*}
so that as a chain in $C_{n-2}(\buildingssymp(\Kn))$, we have $[\apartment_{\mathbf{L}}]=[\apartment_{\mathbf{L}}]^{\on{good}}+[\apartment_{\mathbf{L}}]^{\on{bad}}$. For $n=3$, this is depicted on the left of \cref{fig_motivation}, where the blue triangles are the chambers corresponding to good elements in $\signp$. 

We will show that $\psi_{*}([\apartment_{\mathbf{L}}]^{\on{good}})=[\apartment_{\mathbf{S}}]$ (\cref{im of good}) and $\psi_{*}([\apartment_{\mathbf{L}}]^{\on{bad}})= 0$ (\cref{im of bad}). These imply that $\psi_{*}([\apartment_{\mathbf{L}}])= [\apartment_{\mathbf{S}}]$, which finishes the proof of \cref{prop_homology_surjectivity}.

\subsubsection{The image of $[\apartment_{\mathbf{L}}]^{\on{good}}$}

\begin{lemma}
\label{lem_image_U_sigma}
Let $\sigma\in \signp$ be such that $[i-1]\subset \sigma([i])$ for all $2\leq i\leq n$. Then
\begin{equation*}
	L_{\sigma([i])} = 
	\begin{cases}
		B_i & ,\, \sigma([i]) = [i], \\
		A_{i}^{(j)} & ,\, \text{otherwise.}
	\end{cases}
\end{equation*} 
where in the second case, $j$ is some element in $\ls -i, \pm (i+1), \ldots, \pm n \rs$.
\end{lemma}
\begin{proof}
As $[i-1]\subset \sigma([i])$, we have 
\begin{equation*}
	L_{\sigma([i])} = L_1 \oplus\ldots \oplus L_{i-1} \oplus L_{\sigma([i]) \setminus [i-1]}  = B_{i-1} \oplus L_{\sigma([i]) \setminus [i-1]}.
\end{equation*}
We have $L_{\sigma([i]) \setminus [i-1]} = L_i$ if and only if $\sigma([i]) = [i]$. Hence, the claim follows from  \cref{sat of some frames}.
\end{proof}

Recall that the maximal simplices of $X_n(\class(R))$ are of the form $C_{\mathbf{e}} = (c_1^{\mathbf{e}}, \ldots, c_n^{\mathbf{e}})$ as defined in \cref{sec_defn_X_n}.

\begin{lemma}
\label{lem_image_of_good_chamber}
Let $\mathbf{e} = (\epsilon_1, \ldots, \epsilon_n)\in \ls 0,1 \rs^n$ and $\sigma = s_n^{\epsilon_n}\cdots s_1^{\epsilon_1}$.
Then $\psi(\mathbf{L}_\sigma) = C_\mathbf{e}$.
\end{lemma}
\begin{proof}
The image under $\psi$ of $\mathbf{L}_{\sigma}=  L_{\sigma([1])} \subset \ldots \subset L_{\sigma([n])}$
is given by the simplex
\begin{equation*}
	([L_{\sigma([1])}], \ldots, [L_{\sigma([n])}]).
\end{equation*}
By \cref{lem_descrip_of_good_elements}, for all $2\leq i \leq n$, we have $\sigma([i]) \supset [i-1]$ and $\epsilon_i = 0$ if and only if $\sigma([i]) = [i]$.
Hence by \cref{lem_image_U_sigma}, 
\begin{equation*}
	[L_{\sigma([i])}] = \begin{cases}
		b_i & \epsilon_i=0,\\
		a_i & \epsilon_i = 1,		
	\end{cases}
\end{equation*}
which is what we needed to show.
\end{proof}

\begin{lemma}\label{im of good}
    We have $\psi_{*}([\apartment_{\mathbf{L}}]^{\on{good}})=[\apartment_{\mathbf{S}}]$.
\end{lemma}
\begin{proof}
Let $\sigma = s_n^{\epsilon_n}\cdots s_1^{\epsilon_1} \in \signp$
be a good element. By \cref{lem_image_of_good_chamber}, we have $\psi(\mathbf{L}_\sigma) = C_\mathbf{e}$.
Furthermore, there is an obvious equality $\len(\sigma) = \sum \epsilon_{i}$.
Hence,
\begin{equation*}
	\psi_{*}([\apartment_{\mathbf{L}}]^{\on{good}}) = 
	\sum_{\sigma\in\signp \text{ good}} (-1)^{\len(\sigma)}\mathbf{L}_{\sigma}
	 = \sum_{\mathbf{e}\in \ls 0,1 \rs^n} (-1)^{\sum \epsilon_{i}}C_\mathbf{e} = [\apartment_{\mathbf{S}}].
	 \qedhere
\end{equation*}
\end{proof}

\subsubsection{The image of $[\apartment_{\mathbf{L}}]^{\on{bad}}$}
\label{sec_image_bad_simplices}
Suppose that $\sigma$ is a bad element of $\signp$. 
Then by \cref{lem_descrip_of_good_elements}, there is some $2\leq i \leq n$ such that $[i-1]\not\subset \sigma([i])$.
Let $k_{\sigma}$ be the smallest such integer.

As $k_{\sigma}$ is minimal, we have $[k_\sigma -2] \subset \sigma([k_\sigma-1])$ but $\sigma([k_\sigma-1])\neq [k_\sigma-1]$ and $\sigma(k_{\sigma})\neq k_{\sigma}-1$.
Define $j_\sigma$ to be the unique element in $\sigma([k_\sigma-1])\setminus [k_\sigma -2] $. We have $j_\sigma \neq k_{\sigma}-1$, so 
\begin{equation}
\label{eq_j_sigma}
	j_\sigma \not \in [k_\sigma-1].
\end{equation}

Let $\tau_\sigma\in \signp$ be the element that swaps $j_{\sigma}$ and $\sigma(k_{\sigma})$,
\begin{equation*}
	\tau_\sigma(\pm j_{\sigma}) = \pm \sigma(k_{\sigma}), \, \tau_\sigma(\pm \sigma(k_{\sigma}))=\pm j_{\sigma},\, \tau_\sigma(i) = i \text{ for } i\not\in \ls \pm j_{\sigma}, \pm \sigma(k_{\sigma}) \rs.
\end{equation*}
As $j_\sigma\in \sigma([k_\sigma-1])$, we have $j_\sigma\not =  \sigma(k_\sigma)$.  Hence, $\tau_\sigma$ is a non-trivial element. It is a reflection, i.e.~a conjugate of an element in $S$ (see \cite[Proposition 8.1.5]{BB:CombinatoricsCoxetergroups}). In particular, it has odd word length, so
\begin{equation}
\label{eq_sign_tau_sigma}
(-1)^{\len(\tau_\sigma)} = -1.
\end{equation}

We define a map $\iota$ on the set of bad elements of $\signp$ by
\begin{align*}
\iota(\sigma)\coloneqq \tau_\sigma \sigma.
\end{align*}
In other words, given a bad element $\sigma$, the element $\iota(\sigma)\in \signp$ is obtained by swapping $j_\sigma$ with $\sigma(k_\sigma)$ in the image, i.e.
    \[
    \iota(\sigma)(i)=  
    \begin{cases}
        \sigma(i) & \text{ if } i\neq \sigma^{-1}(j_\sigma ),  k_\sigma,\\
        \sigma(k_{\sigma}) & \text{ if } i=\sigma^{-1}(j_\sigma ),\\
        j_\sigma & \text{ if } i=k_{\sigma }.\\
    \end{cases}
\]

\begin{lemma} 
The map $\iota$ defines an involution on the set of bad elements of $\signp$.
\end{lemma}
\begin{proof}
We have $\iota(\sigma)([k_\sigma]) = \left(\sigma([k_\sigma])\setminus \sigma(k_\sigma)\right) \cup \ls j_\sigma \rs$.
As $j_\sigma$ is not contained in $[k_\sigma-1]$ (see \cref{eq_j_sigma}) and $[k_\sigma-1]\not \subset \sigma([k_\sigma])$, this implies that also $[k_\sigma-1]\not \subset\iota(\sigma)([k_\sigma])$. Hence, $\iota(\sigma)$ is bad and $\iota$ is indeed a map
\begin{equation*}
	\iota\colon \{\sigma\in \signp\colon \sigma \text{ is bad} \}\to \{\sigma\in \signp\colon \sigma \text{ is bad} \}.
\end{equation*}

On the other hand, $\iota(\sigma)$ and $\sigma$ agree on the elements of $[\pm n]$ that are sent to $[k_\sigma-2]$, i.e.~
\begin{equation*}
	P\coloneqq \iota(\sigma)^{-1}([k_\sigma-2]) = \sigma^{-1}([k_\sigma-2]) \text{ and } \iota(\sigma)|_{P} = \sigma|_{P}.
\end{equation*}
So in particular, $[i-1]\subset \iota(\sigma)([i])$ for all $i<k_\sigma$ and 
\begin{equation*}
        \iota(\sigma)([k_\sigma-1])\setminus  [k_\sigma-2] = \{\sigma(k_{\sigma})\}\neq \{k_{\sigma}-1\}.
\end{equation*}
In other words, $k_{\sigma} = k_{\iota(\sigma)}$ and  $j_{\iota(\sigma)}=\sigma(k_{\sigma})$.
This implies that $\iota$ is an involution.
\end{proof}

\begin{lemma}\label{im of bad}
    We have $\psi_{*}([\apartment_{\mathbf{L}}]^{\on{bad}})= 0$.
\end{lemma}    
\begin{proof}
    Let $\sigma\in\signp$ be a bad permutation. Then $\iota(\sigma) = \tau_\sigma \sigma$, where $(-1)^{\len(\tau_\sigma)}= -1$ as observed in \cref{eq_sign_tau_sigma}. Hence
\begin{equation*}
	(-1)^{\len(\iota\sigma)}=(-1)(-1)^{\len(\sigma)}
\end{equation*}
    and it suffices to show that the chambers $\psi(\mathbf{L}_{\sigma})$ and $\psi(\mathbf{L}_{\iota\sigma})$ coincide, i.e.~that $[L_{\sigma([j])}]=[L_{\iota(\sigma)([j])}]$ for all $j=1,\ldots, n$. We consider three cases:

Firstly, for $1\leq j< \sigma^{-1}(j_{\sigma})$, we have $\iota(\sigma)(j)=\sigma(j)$, so $[L_{\sigma([j])}]=[L_{\iota(\sigma)([j])}]$.
Secondly, for $\sigma^{-1}(j_{\sigma}) \leq j< k_\sigma $, we have $\sigma([j])=[j-1]\cup \ls j_\sigma \rs$ and $\iota(\sigma)([j])=[j-1]\cup \ls \sigma(k_\sigma) \rs$. Hence,
\begin{equation*}
	L_{\sigma([j])}=L_1 \oplus\ldots \oplus L_{j-1} \oplus L_{j_{\sigma}} \text{ and } L_{\iota(\sigma)([j])}=L_1 \oplus\ldots \oplus L_{j-1} \oplus L_{\sigma(k_\sigma)}.
\end{equation*}
By \cref{sat of some frames}, these are equal to $A^{(j_{\sigma})}_{j}$ and $A^{(\sigma(k_\sigma))}_{j}$, respectively. Hence 
\begin{equation*}
	[L_{\sigma([j])}] = [A^{(j_{\sigma})}_{j}] = a_j = [A^{(\sigma(k_\sigma))}_{j}] = [L_{\iota(\sigma)([j])}].
\end{equation*}
Thirdly, for $k_\sigma\leq j \leq n$, we have $\sigma([j]) = \iota(\sigma)([j])$, so also $[L_{\sigma([j])}]=[L_{\iota(\sigma)([j])}]$.
\end{proof}

\section{Proofs of the main theorems}
\label{sec_proofs_main_results}
In this section, we prove the main theorems stated in the introduction. Using \cref{prop_homology_surjectivity}, the proof of \cref{thm_non_trivial_cohomology} is almost identical to \cite[Proof of Theorem D']{Church2019}.

\begin{proof}[Proof of \cref{thm_steinberg_coinvariants}]
The map $\psi: \buildingssymp_n(K) \to X_n(\class(R))$ defined in \cref{sec_defn_X_n} is invariant under the action of $\Sp{2n}{R}$. 
Hence, the induced map 
\begin{equation*}
	\psi_{*}\colon \Stsymp = \tilde{H}_{n-1}(\buildingssymp(\Kn))\to \tilde{H}_{n-1}(X_{n}(\hcl))\cong \mbZ^{(|\class(R)|-1)^n}
\end{equation*}
factors through the coinvariants, giving a map
\begin{equation*}
\phi: H_0(\Sp{2n}{R}; \Stsymp) \cong \Stsymp_{\Sp{2n}{R}} \to \mbZ^{(|\class(R)|-1)^n}.
\end{equation*}
By \cref{prop_homology_surjectivity}, $\psi_{*}$ and hence $\phi$ is surjective.
\end{proof}

\begin{corollary}
\label{cor_Q_coinvariants}
Let $R$ be a Dedekind domain. Then $\dim H_0(\Sp{2n}{R}; \Stsymp\otimes_{\mbZ} \mbQ)\geq (|\class(R)|-1)^n$.
\end{corollary}
\begin{proof}
This immediately follows from \cref{thm_steinberg_coinvariants} because
\begin{equation*}
	H_0(\Sp{2n}{R}; \Stsymp\otimes_{\mbZ} \mbQ) \cong (\Stsymp\otimes_{\mbZ} \mbQ)_{\Sp{2n}{R}} \cong \Stsymp_{\Sp{2n}{R}}\otimes_{\mbZ} \mbQ .\qedhere
\end{equation*}
\end{proof}

Both the cohomology vanishing of \cref{thm_non_trivial_cohomology} and the non-integrality of \cref{thm_non_integrality} quickly follow from this corollary:

\begin{proof}[Proof of \cref{thm_non_trivial_cohomology}]
Borel--Serre Duality (\cref{thm_Borel_Serre}) implies that if $K$ is a number field with ring of integers $R = \mcO_K$, then 
\begin{equation*}
H^{\nu_{n}}(\Sp{2n}{R}; \mbQ) \cong H_0(\Sp{2n}{R}; \Stsymp\otimes_{\mbZ} \mbQ).
\end{equation*}
Hence, the result follows from \cref{cor_Q_coinvariants}.
\end{proof}

\begin{proof}[Proof of \cref{thm_non_integrality}]
This follows combining \cref{thm_steinberg_coinvariants} with the results of \cite{Brueck2022c}.
As pointed out in \cite[Remark 3.5]{Brueck2022c}\footnote{\cite{Brueck2022c} works in a more general setup. To specialise to the case here, choose $\mathcal{G} = \on{Sp}_{2n}$. Then the Steinberg module denoted $\texttt{St}$ in \cite{Brueck2022c} becomes $\Stsymp$. The ring $\mathfrak{O}$ in \cite{Brueck2022c} is what we call $R$ here and the ring $R$ in \cite[Remark 3.5]{Brueck2022c} here is $\mbQ$.}, the proof of \cite[Theorem 3.4]{Brueck2022c} shows that if $\Stsymp$ is generated by integral apartment classes, then $H_0(\Sp{2n}{R}; \Stsymp\otimes_\mbZ \mbQ)$ is trivial. By \cref{cor_Q_coinvariants}, this is not possible if $|\class(R)| > 1 $.

Note that unfortunately, \cite[Remark 3.5]{Brueck2022c} is only stated for the case where $R$ is a number ring. But it holds more generally for every Dedekind domain. To verify that this is the case, one easily checks that \cite[Proposition 3.3]{Brueck2022c} holds in that generality because \cite[Theorem 2.5]{Brueck2022c} and \cite[Theorem 2.7]{Brueck2022c} do.
\end{proof}

\section{Open questions}
\label{sec_open_questions}
The present article and the work of Church--Farb--Putman \cite{Church2019} provide non-trivial cohomology in the virtual cohomological dimension for the groups $\Sp{2n}{R}$ and $\SL{n}{R}$ if $R$ is not a PID.
It seems natural to ask to which other types of groups these results could be extended.

A possible class of such groups are arithmetic Chevalley groups.
Let $\chevalley$ be a Chevalley--Demazure group scheme of rank $n$ and $K$ a number field with ring of integers $R = \mcO_K$.
One can associate to $\chevalley(K)$ a spherical building $\Delta(\chevalley(K))$ and Borel--Serre duality shows that the $(n-1)$-st reduced homology of this building is a dualizing module for $\chevalley(R)$.
In the present article, we consider the case $\chevalley = \on{Sp}_{2n}$, where $\building(\Sp{2n}{K}) = \buildingssymp(K^{2n})$ and the dualizing module is the symplectic Steinberg module $\Stsymp_n(K)$. The article \cite{Church2019} studies the case $\chevalley = \on{SL}_{n+1}$, where $\building(\SL{n+1}{K})$ is (the order complex of) the poset of proper subspaces of $K^{n+1}$ and the dualizing module is the Steinberg module $\St_{n+1}(K)$.
Obtaining a non-vanishing result for non-PIDs $R$ and further Chevalley groups would be especially interesting because Br\"uck--Santos Rego--Sroka \cite{Brueck2022c} showed that if $R$ is Euclidean, then for almost all types $\chevalley$, the cohomology of $\chevalley(R)$ vanishes in the virtual cohomological dimension.

If one wanted to apply the same strategy as in \cite{Church2019} and the present article, an important step would be to define a map connecting $\building(\chevalley(K))$ and $X_{n}(\class(R))$, similar to the map $\psi$ defined in \cref{sec_defn_X_n}.
This means finding a reasonable, $\chevalley(R)$-invariant way to associate to parabolic subgroups of $\chevalley(K)$ elements in $\class(R)$.
For this, it might be helpful to observe that $X_{n}(\class(R))$ is itself the quotient of a building by an arithmetic Chevalley group of rank $n$:
\begin{equation*}
	X_{n}(\class(R)) \cong \Delta(\productSL(K))/\productSL(R), \text{ where } \productSL = \underbrace{\on{SL}_{2}\times \cdots \times \on{SL}_{2}}_{n \text{ times}}.
\end{equation*}
It might also be that instead of $X_n(\class(R))$, one would need to consider another finite complex $X(\chevalley(R))$.
In any case, one would probably need to replace the signed permutation group $\signp$ with the Weyl group of $\chevalley$. This step seems less critical to us; note that the definition of ``good'' elements in $\signp$ (\cref{def_good_bad_elements}) is formulated for general Coxeter groups.

The map $\psi\colon \buildingssymp(\Kn) \to X_{n}(\class(R))$ induces $\bar{\psi}\colon \buildingssymp(\Kn)/\Sp{2n}{R} \to X_{n}(\class(R))$.
Using \cref{chain of istropics}, this map is a surjection of cell complexes.
Church--Farb--Putman \cite[Proposition 5.4]{Church2019} showed that analogous map in the setting of $\SL{n}{R}$ is an isomorphism. (They state it for $\GL{n}{R}$, but the quotients $\buildingssln_n(R)/\GL{n}{R}$ and $\buildingssln_n(R)/\SL{n}{R}$ agree.)
We wonder whether this is also the case for $\psi$ or in a more general setting.
So in view of the previous paragraph, the question would be: Do we always get an isomorphism $\Delta(\chevalley(K))/\chevalley(R) \cong \Delta(\productSL(K))/\productSL(R)$? If not, is there some other concrete description of $\Delta(\chevalley(K))/\chevalley(R)$ or at least a lower bound on its number of cells?

It could also be interesting to study the case where $R$ is a PID that is not Euclidean. For $\chevalley = \on{SL}_n$, this was done by Miller--Patzt--Wilson--Yasaki \cite{MPWY:Nonintegralitysome}. But already for $\chevalley = \on{Sp}_{2n}$, we are not aware of results in this direction.

The above questions are all about using similarities between $\on{SL}_n$ and $\on{Sp}_{2n}$ to obtain similar lower bounds for the size of their top-dimensional rational cohomology. Conversely, one could also ask whether using differences between them, i.e.~properties that are more specific to the symplectic group, could yield stronger bounds than the ones we obtained.

\bibliographystyle{halpha}
\bibliography{bibliography}
\end{document}